\documentclass[12pt, reqno]{amsart}
\usepackage{amsmath}
\usepackage{amssymb}
 2

\newtheorem{thm}{Theorem}[section]
\newtheorem{lem}[thm]{Lemma}
\newtheorem{prop}[thm]{Proposition}

\newcommand{\thmref}[1]{Theorem~\ref{#1}}
\newcommand{\propref}[1]{Proposition~\ref{#1}}
\newcommand{\lemref}[1]{Lemma~\ref{#1}}

\theoremstyle{remark}

\newcommand{\Z}{{\mathbb Z}}
\newcommand{\Q}{{\mathbb Q}}
\newcommand{\N}{{\mathbb N}}

\newcommand{\C}{{\mathbb C}}

\begin{document}

\title[Digamma functions and Euler-Lehmer constants]{The Digamma function, 
Euler-Lehmer constants and their $p$-adic counterparts}
\author{T. Chatterjee and S. Gun}
\address{T. Chatterjee \\ $\phantom{mmmmmmmmmmmmmmmm
mmmmmmmmmmmmmmmmmmmmmmm}$
Department of Mathematics,
Queen's University, Kingston,
ON K7L3N6,
Canada.}
\email{tapas@mast.queensu.ca}
\address{S. Gun \\ $\phantom{mmmmmmmmmmmmmmmm
mmmmmmmmmmmmmmmmmmmmmmm}$
Institute for Mathematical Sciences, 
C.I.T Campus, 4th Cross street, 
Taramani,  Chennai, 600 113, 
Tamil Nadu, India.}
\email{sanoli@imsc.res.in}

\renewcommand{\thefootnote}{}
\footnote{ \noindent\textbf{} \vskip0pt
\textbf{2010 Mathematics Subject Classification:} 11J91. 
}

\maketitle

\begin{abstract}
The goal of this article is twofold.  We first extend a result of Murty and Saradha~\cite{MS} 
related to the digamma function at rational arguments. Further, we extend another result 
of the same authors \cite{MS1} about the nature of $p$-adic Euler-Lehmer constants.
\end{abstract}

\smallskip

\section{Introduction}

\medskip

For a real number $x \ne 0, -1, \cdots$,  the digamma function $\psi(x)$ is the logarithmic 
derivative of the  gamma function defined by
$$
- \psi(x)  = \gamma + \frac{1}{x}  + \sum_{n=1}^{\infty} \left( \frac{1}{n+x}  - \frac{1}{n} \right),
$$
where $\gamma$ is Euler's constant. Just like the case of the gamma function,
the nature of the values of the digamma function at algebraic or even rational arguments 
is shrouded in mystery.

In the rather difficult subject of irrationality or transcendence, sometimes it is 
more pragmatic  to look  at a  family of special values 
as opposed to a single specific value and derive something meaningful. An
apt instance here is  the result of Rivoal \cite{TR} about irrationality of infinitude of odd
zeta values as opposed to that of a single specific odd zeta value.

In this context, Murty and Saradha \cite{MS} in a recent work
have made some breakthroughs 
about transcendence of a certain family of digamma values.
In particular, they proved the following.
\begin{thm}[Murty and Saradha]\label{murty} For any 
positive integer $n > 1$, at most one of the $\phi(n) + 1$ numbers
in the set
$$
\{\gamma\} \cup \left\{ \psi(r/n) |  ~ 1 \le  r \le n,  ~(r,n) = 1 \right\}
$$
is algebraic.
\end{thm}
In this article, we extend their result and prove the following.
\begin{thm}\label{CGR}
At most one element in the following infinite set
\begin{equation*}
\{ \gamma\} \cup \left\{\psi(r/n)  ~|~~  n> 1,~  1 \le r < n, ~(r,n)=1
\right\} 
\end{equation*}
is algebraic. 
\end{thm}
In a recent work \cite{GMR}, the question of linear independence
of these numbers is studied. 

In another context, Lehmer \cite{DL} defined generalized 
Euler constants $\gamma(r, n)$ for $r, n \in \N$ with $r \le n$ 
by the formula
$$
\gamma(r, n) = \lim_{ x \to \infty} 
\left( \sum_{m \le x \atop m \equiv r\!\!\!\!\!\pmod{n}} \frac{1}{m}  - \frac{\log x}{n}  \right).
$$
Murty and Saradha, in their papers \cite{{MS}, {MS2}}, investigated the
nature of Euler-Lehmer constants $\gamma(r,n)$ and proved results similar
to \thmref{murty} and \thmref{CGR}.  For an exhaustive
account of Euler's constant, see the recent article of Lagarias \cite{JL}.

From now onwards $p$ and $q$ will always denote prime numbers. 
In another work~\cite{MS1}, Murty and Saradha investigated the $p$-adic 
analog $\gamma_p$ of Euler's constant as well as the
generalized $p$-adic Euler-Lehmer constants $\gamma_p(r,q)$. Here
$r \in \N$ with $1 \le r < q$. 
They also studied the values of the $p$-adic digamma function
$\psi_p(r/p)$ for $1 \le r < p$. 

We shall give the definitions of $\gamma_p, \gamma_p(r,q)$ and $\psi_p(x)$ in 
the next section following Diamond  \cite{JD}.   
Here are the results of Murty and Saradha \cite{MS1}.

\begin{thm}[Murty and Saradha]\label{murty1}
Let $q$ be prime. Then at most one of
$$
\gamma_p, ~ \gamma_p(r,q), ~~~~ 1 \le r < q
$$
is algebraic.
\end{thm}

\begin{thm}[Murty and Saradha]\label{murty2}
The numbers $\psi_p(r/p) + \gamma_p$  are transcendental
for $1 \le r < p$. 
\end{thm}

In this paper, we generalize these results. Let $\mathcal{P}$
denote the set of prime numbers in $\N$.  Here we prove
the following:

\begin{thm}\label{req1}
At most one number in the following set
\begin{equation*}
\{ \gamma_p\} \cup \{ \gamma_p(r , q):~ q \in \mathcal{P},
 ~1 \le r < q / 2 \}
\end{equation*}
is algebraic.
\end{thm}

If we normalize the $p$-adic Euler-Lehmer constants by setting
\begin{equation*}
 \gamma_p^*(r,q)=q\gamma_p(r,q),
\end{equation*}
then we have the following result:
\begin{thm}\label{nor}
All the numbers in the following list
$$
\gamma_p, \phantom{m}
\gamma_p^*(r, q),~ q \in \mathcal{P},
 ~1 \le r < q / 2
$$
are distinct. 
\end{thm}

Using this theorem, we prove the following:
\begin{thm}\label{CGR1}
As before, let $q$ run through the set of all prime numbers. Then
there is at most one pair of repetition among the numbers
\begin{equation*}
\gamma_p,  ~\gamma_p(r,q), ~~~
1 \le r < q / 2. 
\end{equation*}
Also if there is such a repetition, then $\gamma_p$ is transcendental.
\end{thm}

Finally, we prove the following theorem.

\begin{thm}\label{CGR2}
Fix an integer $n>1$. At most one element of the following set  
$$
\left\{ \psi_p(r / p^n) + \gamma_p  ~: ~1 \le r < p^n, ~(r, p)=1  \right\}
$$
is algebraic. Moreover,  
$ \psi_p(r/p)  +  \gamma_p$ are distinct when $1\le r < p/2$.
\end{thm}

\section{Preliminaries}

For all discussions in this section, let us
fix a prime $p$. Let $\overline{\Q}_p$ be a fixed algebraic
closure of $\Q_p$ and $\C_p$ be its completion. We fix an embedding
of $\overline{\Q}$ into $\C_p$. Thus the elements
in the set $\C_p\setminus \overline{\Q}$ are the transcendental 
numbers.

We begin with recalling the notion of $p$-adic logarithms
which are of  primary importance in our context.
For the elements in the open unit ball around $1$, that is
$$
D := \{\alpha \in \C_p : |\alpha - 1|_p < 1\},
$$   
$\log_p\alpha$ is defined using
the formal power series
$$
\log(1+X) = \sum_{n=1}^{\infty}\frac{(-1)^{n+1} X^n}{n}
$$
which has radius of convergence $1$.
To extend this to all of $\C_p^{\times}$, note that
every element $\beta \in \C_p^{\times}$ is uniquely
expressible as 
$$
\beta = p^r w \alpha
$$
where $\alpha \in D$, $r \in \Q$ and $w$ is a root of unity
of order prime to $p$. Here $p^r$ is the positive real $r$-th power
of $p$ in $\overline{\Q}$, embedded in  $\overline{\Q}_p$
from the beginning. With this, one defines
$$
\log_p\beta := \log_p \alpha.
$$
Note that $\log_p \beta = 0$ if and only
if $\beta$ is $p^r$ times a root of unity.
We refer to Washington \cite{LW}, Chapter 5 for details.

We shall now define the $p$-adic analog of the digamma
function following the strategy of Diamond. The idea is 
to first define a suitable analog of the classical $\log$-gamma
function. This is defined, for $x \in \C_p^{\times}$, as
$$
G_p(x) = \lim_{k \to \infty} \frac{1}{p^k} \sum_{n=0}^{p^k-1}
(x+n)\log_p(x+n) -(x+n).
$$
This function satisfies properties analogous to the classical
 $\log$-gamma function; for instance it satisfies
$$
G_p(x+1) = G_p(x) + \log_p x.
$$
It also satisfies an analog of the classical Gauss' identity
(up to a term $\log\sqrt{2\pi}$ ),
namely 
$$
G_p(x) =\left(x-\frac{1}{2}\right) \log_p m 
+  \sum_{a=0}^{m-1} G_p\left(\frac{x+a}{m}\right)
$$
for a positive integer $m$ when the right side is defined.

The $p$-adic digamma function $\psi_p(x)$ is defined as 
the derivative of $G_p(x)$ and hence is given by (for 
$-x \notin \N$),
$$
\psi_p(x)=  \lim_{k \to \infty} \frac{1}{p^k} \sum_{n=0}^{p^k-1}
\log_p(x+n).
$$ 

Recall that the classical generalised Euler constant $\gamma(r,f)$
defined by Lehmer satisfies 
$$
\psi(r/f) = \log f - f \gamma(r,f)
$$
for $1 \leq r \leq f$.

In the $p$-adic set up, one also defines
$\gamma_p(r,f)$ for integers $r,f$ with $f \geq 1$ as follows.
If the $p$-adic valuation $\nu(r/f)$ of $r/f$ is negative, then
$$
\gamma_p(r,f) = - \lim_{k \to \infty} \frac{1}{fp^k} \sum_{m \geq 1,\atop
m\equiv r({\rm mod} f)}^{fp^k-1}
\log_p m.
$$
On the other hand when  $\nu(r/f) \geq 0$, we first write
$f$ as $f=p^kf_1$ with $(p,f_1) = 1$ and then define
$$
\gamma_p(r,f) =  \frac{p^{\varphi(f_1)}}{p^{\varphi(f_1)}-1}
\sum_{n \in N(r,f)} \gamma_p(r+nf,p^{\varphi(f_1)}f) 
$$
where
$$
N(r,f) = \{n~|~0 \leq n < p^{\varphi(f_1)},~nf+r \not\equiv 0~({\rm mod}
~ p^{\varphi(f_1)+k}) \}.
$$
Finally, we set
$$
\gamma_p = \gamma_p(0,1) = - \frac{p}{p-1}
\lim_{k \to \infty} 
\frac{1}{p^k}
\sum_{m \geq 1, \atop (m,p)=1}^{p^k-1}
\log_p m.
$$
 We shall need the following identity of Diamond (see p.334 of
\cite{JD}). 
\begin{thm}
If $q>1$ and $\zeta_q$ is a primitive $q$-th root
of unity, then 
$$
q \gamma_p(r, q) = \gamma_p -
 \sum_{a=1}^{q-1} \zeta_{q}^{-ar} \log_p(1 - \zeta_q^a).
$$
\end{thm}

Let us now state the pre-requisites from transcendence theory.
We shall need the following result of Baker (see p.11 of \cite{AB})
 involving classical
logarithms of complex numbers.

\begin{thm}[Baker]
If $\alpha_1, \cdots, \alpha_n$ are non-zero algebraic numbers
and $\beta_1, \cdots, \beta_n$ are algebraic numbers, then
$$
\beta_1\log\alpha_1 + \cdots + \beta_n\log\alpha_n
$$
is either zero or transcendental.
\end{thm}

We shall need analogous results for linear forms in
$p$-adic logarithms. More precisely, we shall need
the following consequence of a theorem
of Kaufman \cite{RK}
as noticed by Murty and Saradha 
(see p. 357 of \cite{MS1}).

\begin{thm}\label{K}
Suppose that  $\alpha_1 , \cdots, \alpha_m$ are
non-zero algebraic numbers that are multiplicatively
independent over $\Q$
and $\beta_1, \cdots, \beta_m$ are arbitrary 
algebraic numbers {\rm (}not all zero{\rm)}.
Further suppose that 
$$ 
|\alpha_i - 1 | < p^{-c} 
\phantom{m}   \text{ for } ~ 1 \le i \le m,
$$ 
where $c$ is a constant which depends only on the 
degree of the number field
generated by $\alpha_1, \cdots, \alpha_m, \beta_1, \cdots, \beta_m$.
Then
$$
\beta_1 \log_p\alpha_1 + \cdots + \beta_m \log_p\alpha_m
$$
is transcendental.
\end{thm}

\section{Proof of \thmref{CGR}}

In order to prove \thmref{CGR}, we need the following lemmas.
\begin{lem}\label{lem1}
For all $n>1$ and for all $r \in \N$ with $(r, n)=1$ and $1 \le r < n$, 
all the numbers in the following list 
\begin{equation*}
\gamma, ~~\psi(r/n)
\end{equation*}
are distinct.
\end{lem}
\begin{proof}
For a real number $x > 0$, we have
\begin{equation*}
\psi'(x) = \sum_{n = 0}^\infty \frac{1}{(n + x)^2} > 0.
\end{equation*}
Hence $\psi(x)$ is strictly increasing function for $x >0$. 
\end{proof}

\begin{lem}\label{lem2}
For $q > 1$ and $ 1 \le a < q$ with $(a,q) =1$, one has
$$
 - \psi(a/q)  - \gamma  =  \log q  - 
  \sum_{b=1}^{q-1} e^{-2\pi i ba/q} \log(1 - e^{2\pi ib/q}).
$$
\end{lem}
For a proof of this lemma, see page 311 of \cite{MS}.

\smallskip

We are now ready to complete the proof of \thmref{CGR}.
We prove this theorem by the method of contradiction. Suppose
that the above assertion is not true. 
By the work of Murty and Saradha \cite{MS}, it follows that
$\gamma$ and $\psi(a/q)$ for some $1 \le a < q$ with $(a,q)=1$
cannot be both algebraic. So assume that there exists
$1\leq a_1<q_1$ with $(a_1,q_1)=1$ and $1\leq a_2<q_2$
with $(a_2,q_2)=1$ such that both 
$\psi(a_1/q_1)$ and $\psi(a_2/q_2)$  
are algebraic numbers.  Note that by \lemref{lem2}, we have
\begin{eqnarray}\label{eq2}
\psi(a_1/q_1) - \psi(a_2/q_2) &=&   \log\frac{q_2}{q_1}  
- \sum_{b=1}^{q_2-1} e^{-2\pi i ba_2/q_2} \log(1 - e^{2\pi ib/q_2}) \nonumber \\
&& + \phantom{m} \sum_{c=1}^{q_1-1} e^{-2\pi i ba_1/q_1} \log(1 - e^{2\pi ib/q_1}).
\end{eqnarray}
The right hand side is a algebraic linear combination of linear forms 
of logarithms of algebraic numbers. Also by \lemref{lem1},
it is non-zero. Hence by Baker's theorem 
it is transcendental, a contradiction. 
This completes the proof of the theorem.

\section{Proofs of other theorems}

Next, we prove a proposition which will play a pivotal role in proving the rest of the
theorems.

\begin{prop}\label{imp}
For $p_i \in \mathcal{P}$, let $q_i = p_i^{m_i}$, where $m_i \in \N$
and $\zeta_{q_i}$ be a primitive $q_i$-th root of unity.
Then for any finite subset $\rm J$ of $\mathcal{P}$, the numbers
$$
1-\zeta_{q_i},  \phantom{m}  \frac{1-\zeta_{q_i}^{a_i}}{1-\zeta_{q_i}}, 
\text{  where }~1< a_i < \frac{q_i}{2}, ~(a_i,q_i)=1~\text{and } p_i \in {\rm J} ,  
$$
are multiplicatively independent.
\end{prop}

\begin{proof}
Write ${\rm I} = \{ i ~|~ p_i \in {\rm J} \}$.
We will prove this proposition by induction on $|\rm I|$. First suppose that $|\rm I| =1$.
Then the proposition is true by the work of Murty and Saradha (see page 357
of \cite{MS1}). Next suppose that the proposition is true for all $\rm I$ with 
$ |\rm I| <  n$. Now suppose that $|\rm I | =n$.
Note that for any $i$, the numbers
\begin{equation}\label{obs}
\frac{1-\zeta_{q_i}^{a_i}}{1-\zeta_{q_i}}, \text{  where } 
1<a_i<\frac{q_i}{2}, ~(a_i,q_i)=1
\end{equation}
are multiplicatively independent units in $\Q(\zeta_{q_i})$ 
(see page 144 of \cite{LW}).  
Suppose that there exist integers $\alpha_i, \beta_{a_i}$ for 
$i \in {\rm I}$ and with $a_i$ as in the lemma 
such that
\begin{equation}\label{rel}
\prod_{i \in {\rm I}} \left\{(1 - \zeta_{q_i})^{\alpha_i}
{\prod_{\substack{1< a_i  < q_i/2 \\
                              (a_i,q_i)=1}}
\left(\frac{1-\zeta_{q_i}^{a_i}}
{1-\zeta_{q_i}}\right)^{\beta_{a_i}} }\right\} ~= ~1. 
\end{equation}
Taking norm on both sides, we get
\begin{equation*}
\prod_{i \in {\rm I}}p_{i}^{\alpha_i A_i} ~=~ 1, \text{  where  } A_i \ne 0, ~A_i \in \N. 
\end{equation*}
Since $p_i$'s are distinct primes, we have $\alpha_i=0$ for all $i \in  {\rm I}$. Thus \eqref{rel} 
reduces to 
\begin{equation}\label{1}
\prod_{i \in {\rm I}}
\prod_{\substack{1< a_i <  {q_i} /2 \\ (a_i,q_i)=1}}
\left(\frac{1-\zeta_{q_i}^{a_i}}
{1-\zeta_{q_i}}\right)^{\beta_{a_i}} ~=~ 1. 
\end{equation}
Since $|\rm I| > 1$, there exists $i_1, i_2 \in {\rm I}$ with $i_1 \ne i_2$ such that
\begin{equation*}
\prod_{i \in {\rm I},\atop
             i \ne i_1}\prod_{\substack{1< a_i < q_i/2\\ 
(a_i,q_i)=1}}\left(\frac{1-\zeta_{q_i}^{a_i}}{1-\zeta_{q_i}}\right)^{\beta_{a_i}}
~=~ \prod_{\substack{1< a_{i_1}  < q_{i_1}/2\\ 
(a_{i_1},q_{i_1})=1}}\left(\frac{1-\zeta_{q_{i_1}}^{a_{i_1}}}
{1-\zeta_{q_{i_1}}}\right)^{-\beta_{a_{i_1}}}. 
\end{equation*}
Note that the left hand side of the above equation
belongs to the number field $\Q( \zeta_{\delta})$, where
 $ \delta = \prod_{ i \in {\rm I} \setminus \{i_1\} } {q_i}$ 
whereas the right hand side belongs
to $\Q(\zeta_{q_{i_1}})$. Since
$$
\Q(\zeta_{\delta}) \cap \Q(\zeta_{q_{i_1}})
= \Q,
$$ 
we see that both sides of the above equation is a rational number having norm $1$. 
Thus we have
\begin{equation*}
\prod_{i \in {\rm I} \setminus \{i_1\}}\prod_{\substack{1< a_i < q_i / 2\\ (a_i,q_i)=1}}
\left(\frac{1-\zeta_{q_i}^{a_i}}{1-\zeta_{q_i}}\right)^{\beta_{a_i}}
= \prod_{\substack{1< a_{i_1} < q_{i_1}/2\\  (a_{i_1},q_{i_1})=1}}
 \left(\frac{1-\zeta_{q_{i_1}}^{a_{i_1}}}{1-\zeta_{q_{i_1}}}
 \right)^{-\beta_{a_{i_1}}} ~=~ \pm 1.  
\end{equation*}
Squaring both sides, we get
\begin{equation*}
\prod_{i \in {\rm I} \setminus \{i_1\}}\prod_{\substack{1< a_i < q_i /2\\
 (a_i,q_i)=1}}\left(\frac{1-\zeta_{q_i}^{a_i}}{1-\zeta_{q_i}}\right)^{2\beta_{a_i}}
=\prod_{\substack{1< a_1< q_{i_1} /2\\
 (a_{i_1},q_{i_1})=1}}\left(\frac{1-\zeta_{q_{i_1}}^{a_{i_1}}}
 {1-\zeta_{q_{i_1}}}\right)^{-2\beta_{a_{i_1}}} ~=~ 1. 
\end{equation*}
Using \eqref{obs}, we see that 
$\beta_{a_{i_1}} = 0$ for all 
$1< a_{i_1} < q_{i_1} / 2$ and $(a_{i_1}, q_{i_1})= 1$.
Then \eqref{1} reduces to  
\begin{equation*}
 \prod_{i \in {\rm I} \setminus {i_1}}\prod_{\substack{1 < a_i < q_i / 2 \\ 
(a_i,q_i)=1}}
\left(\frac{1-\zeta_{q_i}^{a_i}}{1-\zeta_{q_i}}\right)^{\beta_{a_i}}=1. 
\end{equation*}
Now by induction hypothesis, we have $\beta_{a_i} =0$
for all $a_i$ and $i \in {\rm I} \setminus \{i_1\}$. 
This completes the proof of the lemma.
\end{proof}
Using the above \propref{imp} and \thmref{K}, we can prove the
following statement.

\begin{lem}\label{imp1}
Let $ {\rm J}$ be any finite subset of $\mathcal{P}$.
For $q \in {\rm J}$ and $1 < a < q/2$, let
$s_q, t_q^a$ be arbitrary algebraic numbers, not all zero.
Further, let $t_q^a$ be not all zero when $p \in {\rm J}$.
Then 
$$
\sum_{q \in {\rm J}} s_q  ~\log_p(1-\zeta_{q})   ~+~ 
\sum_{q \in {\rm J} \atop 1< a < \frac{q}{2} } t_q^a  ~\log_p 
\left(\frac{1-\zeta_{q}^{a}}{1-\zeta_{q}}\right) 
$$
is transcendental.
\end{lem}

\begin{proof}
Write $\delta = \prod_{q \in {\rm J}} q$.
For any $\alpha \in \Z[\zeta_{\delta}]$ with $p \nmid \alpha$ and $M \in \N$, one has
$$
| \alpha^A - 1 | < p^{-M}
$$
for some $A \in \N$. By choosing $M$ sufficiently large and using \thmref{K}
and \propref{imp}, we get the lemma.
\end{proof}

\begin{lem}\label{req}
Let $q_1, q_2$ be two distinct prime numbers and 
$1 \le r_i < q_i$ for $i=1,2$. Then
\begin{equation}\label{trans}
\sum_{b =1}^{q_2-1}\zeta_{q_2}^{- b r_2}\log_p(1-\zeta_{q_2}^{b}) 
- \sum_{a=1}^{q_1-1}\zeta_{q_1}^{-a r_1}\log_p(1-\zeta_{q_1}^{a})
\end{equation}
is transcendental.
\end{lem}

\begin{proof}
For any $q > 1$ and $(r, q)=1$, we know that
\begin{equation*}
\sum_{a=1}^{q-1}\zeta_q^{-ar} = -1.
\end{equation*}
Hence \eqref{trans} can be written as 
\begin{eqnarray}\label{rel1}
 \log_p\frac{(1-\zeta_{q_1})}{(1-\zeta_{q_2})}  
  - \sum_{a =1}^{q_1-1}\zeta_{q_1}^{-a r_1}\log_p
\left(\frac{1-\zeta_{q_1}^{a}}{1-\zeta_{q_1}}\right)
 +  \sum_{b=1}^{q_2-1}\zeta_{q_2}^{- b r_2}\log_p
 \left(\frac{1-\zeta_{q_2}^{ b}}{1-\zeta_{q_2}}\right). 
\end{eqnarray}
It is because
\begin{equation*}
(1-\zeta_q^{-t})=-\zeta_q^{-t}(1-\zeta_q^t)
\end{equation*}
for any $t \in \N$ and the $p$-adic logarithm is zero on roots of unity, 
we have
\begin{equation*}
\log_p(1-\zeta_q^{-t})=\log_p(1-\zeta_q^t). 
\end{equation*}
Note that the summands in \eqref{rel1} for $a=1, b = 1$ 
and $a = q_1 -1, b = q_2 -1$
are zero. Now pairing up $a$ with $-a$ and $b$ with $-b$ 
in \eqref{rel1}, we get
\begin{eqnarray*}
\log_p\frac{(1-\zeta_{q_1})}{(1-\zeta_{q_2})}
 ~-~   \sum_{1 < a < q_1/2 } \alpha_a
\log_p\left(\frac{1-\zeta_{q_1}^{a}}{1-\zeta_{q_1}}\right) 
~+~  \sum_{1<  b < q_2/2 } \beta_b \log_p 
\left(\frac{1-\zeta_{q_2}^{b}}{1-\zeta_{q_2}}\right),
\end{eqnarray*}
where $\alpha_a = (\zeta_{q_1}^{-a r_1}  +  \zeta_{q_1}^{ a r_1}), 
\beta_b = (\zeta_{q_2}^{-b r_2}  +  \zeta_{q_2}^{ b r_2})$ are 
non-zero algebraic numbers. Hence
using \lemref{imp1}, we deduce that \eqref{rel1}
is transcendental. 
\end{proof}

\subsection{Proof of \thmref{req1}}

We now complete the proof of \thmref{req1}.
Suppose that two of the numbers from the above set
are algebraic. It follows from the works of Murty and Saradha
(see page 351 of \cite{MS1}) that one of them cannot be equal to
$\gamma_p$ or both of them can not be of the form $\gamma_p(r_1, q)$
and $\gamma_p(r_2, q)$. 

Without loss of generality, we can assume that these two numbers
are of the form $\gamma_p(r_1,q_1)$, where $1 \le r_1 < q_1$ and 
$\gamma_p(r_2,q_2)$, where $1 \le r_2 < q_2$ and $q_1 \ne q_2$.
Then $q_1\gamma_p(r_1,q_1)-q_2\gamma_p(r_2,q_2)$ is algebraic. 
Now by Diamond's theorem (see theorem $18$ of \cite{JD}), we have
\begin{eqnarray}\label{right}
&& q_1\gamma_p(r_1,q_1)-q_2\gamma_p(r_2,q_2) \nonumber \\
&=&  -  \sum_{a=1}^{q_1-1}\zeta_{q_1}^{-a r_1}\log_p(1-\zeta_{q_1}^{a})
~+~ \sum_{b =1}^{q_2-1}\zeta_{q_2}^{- b r_2}\log_p(1-\zeta_{q_2}^{b}). 
\end{eqnarray} 
The left hand side is algebraic by assumption whereas the
right hand side is transcendental by \lemref{req}, a contradiction.
This completes the proof of the theorem.

\subsection{Proof of \thmref{nor}}

It follows from Diamond's theorem \cite{JD} that 
\begin{eqnarray}\label{mod}
 \gamma_p^*(r_1, q)- \gamma_p &=&  
 - \sum_{a=1}^{q-1}\zeta_q^{-ar} \log_p(1-\zeta_q^a) \nonumber \\
\text{ and }  \phantom{m} 
\gamma_p^*(r_1, q)- \gamma_p^*(r_2, q)  
&=& 
 - \sum_{a=1}^{q-1}(\zeta_q^{-ar_1} -  \zeta_q^{-ar_2}) 
 \log_p(1-\zeta_q^a)
\end{eqnarray}
where $1 \le r_1, r_2 < q / 2$ and $r_1 \ne r_2$.
Transcendence of the first number follows from the works
of Murty and Saradha (see page 358 of \cite{MS2}) while
that of the second one follows
from the fact that 
$\zeta^{-ar_1} + \zeta^{ar_1} \ne \zeta^{-ar_2} + \zeta^{-ar_2}$
when $1 \le a, r_1, r_2 < q / 2$ with $r_1 \ne r_2$ .
Again by Diamond's theorem, we have
$$
\gamma_p^*(r_1,q_1)- \gamma_p^*(r_2,q_2)  = 
-  \sum_{a=1}^{q_1-1}\zeta_{q_1}^{-a r_1}\log_p(1-\zeta_{q_1}^{a})
~+~ \sum_{b =1}^{q_2-1}\zeta_{q_2}^{- b r_2}\log_p(1-\zeta_{q_2}^{b}),
$$
where $q_1 \ne q_2$.  By \lemref{req},  we know that this number is
transcendental and hence non-zero. This completes the proof.

\subsection{Proof of \thmref{CGR1}}

We will prove this theorem by contradiction.
First note that it is impossible to have 
$$
\gamma_p(r_1,q) = \gamma_p = \gamma_p(r_2, q),
$$
as otherwise $\gamma_p^*(r_1,q) =   \gamma_p^*(r_2, q)$, a contradiction
to \thmref{nor}. Next assume that 
$$
\gamma_p(r_1,q_1) = \gamma_p = \gamma_p(r_2, q_2),
$$
where $q_1 \ne q_2$ and $ 1 \le r_i < q_i / 2$.  Using Diamond's theorem, we can write
\begin{equation}\label{const}
(q_i - 1) \gamma_p =  
 \log_p(1-\zeta_{q_i})
~-~ \sum_{1< a < q_i /2}(\zeta_{q_i}^{ar_i} + \zeta_{q_i}^{-ar_i})
\log_p\left(\frac{1 - \zeta_{q_i}^{a}}{ 1 - \zeta_{q_i}} \right), 
\end{equation}
where $i = 1, 2$.  From this, we get
\begin{eqnarray*}
&&  0 = (q_2 - 1) \log_p(1-\zeta_{q_1})  ~-~ (q_1 - 1)  \log_p(1-\zeta_{q_2}) \\
&&  -  \sum_{1< a < q_1 /2} (q_2 -1 )
(\zeta_{q_1}^{ar_1} + \zeta_{q_1}^{-ar_1})
\log_p\left(\frac{1 - \zeta_{q_1}^{a}}{ 1 - \zeta_{q_1}} \right) \\
&&  +  \sum_{1< a < q_2 /2} (q_1 -1 )
(\zeta_{q_2}^{ar_2} + \zeta_{q_2}^{-ar_2})
\log_p\left(\frac{1 - \zeta_{q_2}^{a}}{ 1 - \zeta_{q_2}} \right) \\
\end{eqnarray*}
which is transcendental by \lemref{imp1}, a contradiction. 
Now suppose that 
\begin{eqnarray*}
\gamma_p(r_1,q_1) = \gamma_p(r_2,q_2) 
\text{      and       }
 \gamma_p(r_3,q_3) = \gamma_p(r_4,q_4)
\end{eqnarray*}
for some $1 \le r_i < q_i / 2,   1 \le i \le 4$. We may assume that $q_1\neq q_2$.
For if $q_1=q_2=q$ (say), then  $\gamma_p(r_1,q)=\gamma_p(r_2,q)$
implies $\gamma_p^*(r_1,q)=\gamma_p^*(r_2,q)$, a contradiction to 
\thmref{nor}.  Similarly, we may assume that $q_3 \ne q_4$. Now
from equation $\gamma_p(r_1,q_1)=\gamma_p(r_2,q_2)$, we deduce that
\begin{eqnarray}\label{const1}
(q_1 - q_2) \gamma_p  
&=&  q_2\log_p(1-\zeta_{q_1}) ~-~  q_1\log_p(1-\zeta_{q_2}) \nonumber \\
&-&  
q_2 \sum_{1<a_1<q_1/2}(\zeta_{q_1}^{a_1r_1} +   \zeta_{q_1}^{-a_1r_1})
 \log_p\left(\frac{1-\zeta_{q_1}^{a_1}}{1-\zeta_{q_1}}\right) \nonumber \\
& + &
q_1\sum_{1<a_2<q_2/2}(\zeta_{q_2}^{a_2r_2}
+ \zeta_{q_2}^{-a_2r_2})
\log_p\left(\frac{1-\zeta_{q_2}^{a_2}}{1-\zeta_{q_2}}\right). 
\end{eqnarray}
Similarly from equation 
$\gamma_p(r_3,q_3) = \gamma_p(r_4,q_4)$, 
we deduce that 
\begin{eqnarray*} 
(q_3-q_4)\gamma_p =
q_4\log_p(1-\zeta_{q_3})&-&q_4\sum_{1<a_3<q_3/2}(\zeta_{q_3}^{a_3r_3}
+ \zeta_{q_3}^{-a_3r_3})\log_p\left(\frac{1-\zeta_{q_3}^{a_3}}{1-\zeta_{q_3}}\right) \\
- q_3\log_p(1-\zeta_{q_4})&+&q_3\sum_{1<a_4<q_4/2}(\zeta_{q_4}^{a_4r_4}
+\zeta_{q_4}^{-a_4r_4})\log_p\left(\frac{1-\zeta_{q_4}^{a_4}}{1-\zeta_{q_4}}\right). 
\end{eqnarray*}
Eliminating $\gamma_p$ from the above two equations, we get
\begin{eqnarray*} 
&& 
0 = (q_3 - q_4) q_2 \log_p(1-\zeta_{q_1})   
- (q_3 - q_4) q_1 \log_p(1-\zeta_{q_2}) \nonumber \\
&& ~~ - (q_1 -  q_2) q_4  \log_p(1-\zeta_{q_3}) 
 +  (q_1 -  q_2) q_3  \log_p(1-\zeta_{q_4})  \nonumber \\
&&  ~~+ ~~ \sum_{i=1}^{4}\sum_{1<a_i<q_i/2}c_{a_i}
\log_p\left(\frac{1-\zeta_{q_i}^{a_i}}{1-\zeta_{q_i}} \right),  
\end{eqnarray*}
where $c_{a_i}$ are algebraic numbers for all $1 \le i \le 4$ and $1< a_i < q_i/2$.
Again using \lemref{imp1}, the number is transcendental.
This completes the proof.

To prove the second part of the theorem, suppose that 
$$
 \gamma_p(r_1,q_1) = \gamma_p  \text{     or     } \gamma_p(r_1,q_1) = \gamma_p(r_2,q_2),
$$ 
where $q_1 \ne q_2$ and $1 \le r_i < q_i/2$ for $i=1, 2$. 
Again, we deduce the result from \eqref{const} or \eqref{const1} using
\lemref{imp1}.

\subsection{Proof of \thmref{CGR2}}

Write
$$
{\rm S} = \left\{ \psi_p(r / p^n)  + \gamma_p   ~:~  1 \le r < p^n,  ~(r, p)=1 \right\}. 
$$  
Suppose that $a, b \in {\rm S}$  be distinct and algebraic. Suppose that 
$$
(a, b) = (\psi_p(r_1/p^n) +  \gamma_p ,~\psi_p(r_2/p^n) + \gamma_p).
$$
Using Diamond's theorem, we have
\begin{eqnarray*}
&& \psi_p(r_1/p^n) -  \psi_p(r_2/p^n) \\
&=&   \sum_{a=1}^{p^n-1} \zeta^{-ar_1} \log_p(1 - \zeta_{p^n}^a) 
 - \sum_{a=1}^{p^n-1} \zeta^{-ar_2} \log_p(1 - \zeta_{p^n}^a)\\
 &=&  \sum_{ 1 < a < p^n /2 \atop (a,p)=1}  \alpha_a
 \log_p \left( \frac{1 - \zeta_{p^n}^{a}}{1 - \zeta_{p^n}} \right),
\end{eqnarray*}
where $\alpha_a$'s are algebraic numbers. But by \lemref{imp1},  this  is 
necessarily transcendental, a contradiction. 

Moreover when $n=1$, we have
$$
\psi_p(r_1/p) -  \psi_p(r_2/p) = 
\sum_{ 1 < a < p /2 }  (\zeta^{-ar_1} + \zeta^{ar_1} -  \zeta^{-ar_2} - \zeta^{ar_2} ) 
 \log_p \left( \frac{1 - \zeta_{p}^{a}}{1 - \zeta_{p}} \right).
$$
Since $ 1\le r_1, r_2 < p/2$, the above linear form in logarithm is transcendental
by \lemref{imp1} and hence non-zero. This completes the proof.

\bigskip
\noindent
{\bf Acknowledgements.} It is our pleasure to thank Ram Murty for going through
an earlier version of the paper. We would also like to thank Purusottam Rath 
for many helpful discussions. The second author
would like to thank ICTP for the hospitality where the final part of the work
was done.

\end{document}